\documentclass[11pt,reqno]{amsart}
\usepackage[utf8]{inputenc}

\usepackage[margin=1in]{geometry}
\usepackage{amsmath,amssymb,amsthm}
\usepackage{mathtools}
\usepackage[T1]{fontenc}
\usepackage{lmodern}
\usepackage{microtype}
\usepackage{xcolor}
\usepackage{hyperref}

\hypersetup{
  colorlinks=true,
  linkcolor=blue!50!black,
  citecolor=blue!50!black,
  urlcolor=blue!50!black
}

\newtheorem{theorem}{Theorem}
\newtheorem{lemma}{Lemma}
\newtheorem{corollary}{Corollary}
\theoremstyle{definition}
\newtheorem{conjecture}{Conjecture}
\theoremstyle{remark}
\newtheorem*{remark}{Remark}

\title[Zeros of the tempered xi function]{Infinitely Many Off-Critical-Line Zeros of the Tempered Xi Function: A Disproof of Yang's Conjecture}

\author{Aiken Kazin}
\address{SDU University, Kaskelen, Kazakhstan}
\email{aiken.kazin@sdu.edu.kz}

\author{Shirali Kadyrov}
\address{Ajman University, Ajman, United Arab Emirates}
\email{s.kadyrov@ajman.ac.ae}

\date{September 24, 2026}

\begin{document}

\begin{abstract}
Yang introduced a tempered xi function $\widehat\xi(s)$ by replacing the hyperbolic cosine in a classical
integral representation of the Riemann xi function by a hyperbolic sine, and conjectured that every zero
of $\widehat\xi$ lies on the critical line $\Re s=1/2$. We disprove this conjecture. After the change of
variables $x=e^{2t}$, one has
\[
\widehat\xi\!\left(\tfrac12+iz\right)=iS(z),
\qquad
S(z)=\int_0^\infty K(t)\sin(zt)\,dt,
\]
where $K$ is positive, smooth, and doubly exponentially decaying. Two integrations by parts give
$S(y)=K(0)/y+O(y^{-2})$ for real $|y|\to\infty$, with $K(0)>0$; hence $S$ has only finitely many real
zeros. On the other hand $S$ is an entire function of order at most one. If it had only finitely many
zeros in the complex plane, Hadamard factorization would force $S(z)=P(z)e^{az+b}$, which is incompatible
with $S(y)\to0$ as $y\to\pm\infty$. Thus $S$ has infinitely many nonreal zeros, and consequently
$\widehat\xi$ has infinitely many zeros off the critical line. The proof is unconditional and does not
use the Riemann Hypothesis or numerical zero finding. A short numerical illustration is included.
\end{abstract}

\maketitle
\section{Introduction}\label{sec:intro}

Let
\[
\Phi(x)=\sum_{n=1}^\infty e^{-\pi n^2x},
\qquad
M(x)=x^{-1/4}\frac{d}{dx}\!\left(x^{3/2}\Phi'(x)\right),
\qquad x\ge1.
\]
A classical integral representation of the Riemann xi function is
\begin{equation}
\xi(s)=4\int_1^\infty M(x)
\cosh\!\left(\frac12\left(s-\frac12\right)\log x\right)dx,
\label{eq:xi-classical}
\end{equation}
see, for example, Edwards~\cite{edwards1974} or Titchmarsh and Heath-Brown~\cite{titchmarsh1986}.

Yang~\cite{yang2024} defined the \emph{tempered xi function}
\begin{equation}
\widehat\xi(s)=4\int_1^\infty M(x)
\sinh\!\left(\frac12\left(s-\frac12\right)\log x\right)dx,
\qquad s\in\mathbb C,
\label{eq:def}
\end{equation}
and proposed the following conjecture.

\begin{conjecture}[Yang~\cite{yang2024}, Conjecture 1]\label{conj:yang}
Every zero of $\widehat\xi(s)$ lies on the critical line $\Re s=1/2$.
\end{conjecture}

The same paper gives equivalent formulations in terms of a Fourier sine transform and an even quotient
obtained after removing the central zero. Subsequent work continued to present the corresponding
real-zero question as open; in particular, Yang's 2026 invited review still lists the relevant
Fourier-sine zero problem among the challenges connected with the heat equation and the Riemann xi
function~\cite{yang2026}. We are not aware of an earlier proof or disproof of
Conjecture~\ref{conj:yang}; thus, to the best of our knowledge, the argument below gives the first
disproof.

Our main result is stronger than the existence of a single counterexample.

\begin{theorem}\label{thm:main}
The tempered xi function $\widehat\xi$ has infinitely many distinct zeros satisfying
\[
\Re s\ne\frac12.
\]
Moreover, only finitely many zeros of $\widehat\xi$ lie on the critical line, and $s=1/2$ is a simple
zero.
\end{theorem}

The proof is elementary once the correct Fourier-sine representation is isolated. The key point is an
endpoint term that has no analogue in the cosine transform relevant to the Riemann Hypothesis.

\section{The sine-transform representation}\label{sec:kernel}

Put $a_{n}=\pi n^2$. Since the defining theta series and all of its derivatives converge uniformly for
$x\ge1$, termwise differentiation gives
\begin{equation}
M(x)=\sum_{n=1}^\infty
\left(a_{n}^2x^{5/4}-\frac{3}{2}a_{n}x^{1/4}\right)e^{-a_{n}x}.
\label{eq:Mseries}
\end{equation}
For $x=e^{2t}$ define
\begin{equation}
K(t):=8e^{2t}M(e^{2t})
=8\sum_{n=1}^\infty
\left(a_{n}^2e^{9t/2}-\frac{3}{2}a_{n}e^{5t/2}\right)e^{-a_{n}e^{2t}},
\qquad t\ge0.
\label{eq:K}
\end{equation}
Then~\eqref{eq:def} becomes
\begin{equation}
\widehat\xi\!\left(\frac12+w\right)
=\int_0^\infty K(t)\sinh(wt)\,dt.
\label{eq:xihatK}
\end{equation}
Define
\begin{equation}
S(z):=\int_0^\infty K(t)\sin(zt)\,dt.
\label{eq:S}
\end{equation}
With $w=iz$ in~\eqref{eq:xihatK},
\begin{equation}
\widehat\xi\!\left(\frac12+iz\right)=iS(z).
\label{eq:transport}
\end{equation}
Thus the critical line for $\widehat\xi$ corresponds exactly to the real axis for $S$.

\begin{lemma}\label{lem:kernel}
The function $K$ is $C^\infty$ and strictly positive for $t\ge0$. For every integer $j\ge0$ there is
$C_j>0$ such that
\begin{equation}
|K^{(j)}(t)|\le C_j
\exp\!\left(\left(\frac92+2j\right)t-\pi e^{2t}\right),
\qquad t\ge0.
\label{eq:decay}
\end{equation}
In particular, every derivative of $K$ is integrable and tends to zero at infinity.
\end{lemma}

\begin{proof}
The $n$th summand in~\eqref{eq:K} is
\[
8a_{n}e^{5t/2}\left(a_{n}e^{2t}-\frac{3}{2}\right)e^{-a_{n}e^{2t}},
\]
which is strictly positive because $a_{n}e^{2t}\ge\pi>3/2$. After $j$ differentiations, every resulting
term is bounded in absolute value by a constant times
\[
n^{4+2j}e^{(9/2+2j)t}e^{-\pi n^2e^{2t}}.
\]
Since, for $u=e^{2t}\ge1$,
\[
\sum_{n=1}^\infty n^{4+2j}e^{-\pi n^2u}
\le e^{-\pi u}\sum_{n=1}^\infty n^{4+2j}e^{-\pi(n^2-1)},
\]
the differentiated series converges uniformly for $t\ge0$ after multiplication by the displayed
majorant, and~\eqref{eq:decay} follows. The remaining assertions are immediate.
\end{proof}

\begin{lemma}\label{lem:entire}
The function $S$ is an odd entire function of order at most one. Its zero at $z=0$ is simple.
\end{lemma}

\begin{proof}
For $|z|\le r$,
$|\sin(zt)|\le e^{rt}$. Lemma~\ref{lem:kernel} therefore gives an integrable majorant, locally uniformly
in $z$, and the same is true after differentiation with respect to $z$. Hence $S$ is entire, and oddness
is immediate from~\eqref{eq:S}. Furthermore,
\[
S'(0)=\int_0^\infty tK(t)\,dt>0,
\]
so the zero at the origin is simple.

It remains to record the growth. By Lemma~\ref{lem:kernel}, for $|z|\le r$ and $r\ge2$,
\[
|S(z)|\le C\int_0^\infty e^{(r+9/2)t-\pi e^{2t}}\,dt
=\frac C2\int_1^\infty u^{r/2+5/4}e^{-\pi u}\,du.
\]
The last integral is bounded by a constant multiple of
$\pi^{-q}\Gamma(q)$ with $q=r/2+9/4$. Stirling's formula then gives
\begin{equation}
\log\max_{|z|\le r}|S(z)|=O(r\log r).
\label{eq:growth}
\end{equation}
Consequently the order of $S$ is at most one.
\end{proof}

\section{Only finitely many zeros on the real axis}\label{sec:realzeros}

\begin{lemma}[Endpoint asymptotic]\label{lem:asymptotic}
For real $y\ne0$,
\begin{equation}
S(y)=\frac{K(0)}{y}
-\frac1{y^2}\int_0^\infty K''(t)\sin(yt)\,dt,
\label{eq:ibp}
\end{equation}
where
\begin{equation}
K(0)=8\sum_{n=1}^\infty
\left(\pi^2n^4-\frac32\pi n^2\right)e^{-\pi n^2}>0.
\label{eq:K0}
\end{equation}
In particular,
\begin{equation}
S(y)=\frac{K(0)}{y}+O(|y|^{-2})
\qquad (|y|\to\infty).
\label{eq:asy}
\end{equation}
\end{lemma}

\begin{proof}
Two integrations by parts are justified by Lemma~\ref{lem:kernel}:
\begin{align*}
S(y)
&={\left[-\frac{K(t)\cos(yt)}{y}\right]}_{0}^{\infty}
+\frac1y\int_0^\infty K'(t)\cos(yt)\,dt\\
&=\frac{K(0)}{y}
-\frac1{y^2}\int_0^\infty K''(t)\sin(yt)\,dt.
\end{align*}
Since $K''\in L^1(0,\infty)$, the remainder is $O(y^{-2})$. Every summand in~\eqref{eq:K0} is positive
because $\pi n^2>3/2$.
\end{proof}

\begin{corollary}\label{cor:real}
The function $S$ has only finitely many real zeros. Equivalently, $\widehat\xi$ has only finitely many
zeros on $\Re s=1/2$.
\end{corollary}

\begin{proof}
Let $B=\int_0^\infty|K''(t)|\,dt$. From~\eqref{eq:ibp},
\[
S(y)\ge \frac{K(0)}{y}-\frac{B}{y^2}>0
\qquad\text{for }y>B/K(0).
\]
By oddness, $S(y)<0$ for $y<-B/K(0)$. Hence all real zeros lie in a compact interval. Since $S$ is a
nonzero entire function, its zeros cannot accumulate in the finite plane, and therefore there are only
finitely many of them.
\end{proof}

\section{Infinitely many complex zeros}\label{sec:complexzeros}

We use only the standard finite-order form of Hadamard's factorization theorem; see, for example,
Boas~\cite[Chapter~2]{boas1954}.

\begin{lemma}\label{lem:hadamard}
The entire function $S$ has infinitely many distinct zeros in $\mathbb C$.
\end{lemma}

\begin{proof}
Assume, to the contrary, that $S$ has only finitely many distinct zeros. By
Lemma~\ref{lem:entire}, $S$ has finite order at most one. Hadamard factorization therefore yields
\begin{equation}
S(z)=P(z)e^{az+b},
\label{eq:factor}
\end{equation}
where $P$ is a nonzero polynomial and $a,b\in\mathbb C$.

On the other hand, Lemma~\ref{lem:asymptotic} gives $S(y)\to0$ as $y\to+\infty$ and as
$y\to-\infty$. This is impossible for~\eqref{eq:factor}: if $\Re a>0$, then $|S(y)|$ cannot tend to zero
as $y\to+\infty$; if $\Re a<0$, it cannot tend to zero as $y\to-\infty$; and if $\Re a=0$, then
$|S(y)|=e^{\Re b}|P(y)|$, which also cannot tend to zero in both directions. This contradiction proves
the lemma.
\end{proof}

\begin{proof}[Proof of Theorem~\ref{thm:main}]
Lemma~\ref{lem:entire} and~\eqref{eq:transport} show that $\widehat\xi$ is entire and that $s=1/2$ is a
simple zero. By Corollary~\ref{cor:real}, only finitely many zeros of $S$ are real, whereas by
Lemma~\ref{lem:hadamard}, $S$ has infinitely many distinct zeros in $\mathbb C$. Hence infinitely many
zeros of $S$ are nonreal. Under the affine map $s=1/2+iz$, these become zeros of $\widehat\xi$ with
$\Re s\ne1/2$. Therefore Conjecture~\ref{conj:yang} is false.
\end{proof}

\begin{remark}
The disproof is driven by the nonzero endpoint value $K(0)$. For the sine transform, integration by
parts produces the leading term $K(0)/y$, which eventually fixes the sign on the real axis. The
classical xi function instead involves the cosine transform of the same kernel, for which this endpoint
term is absent. Thus the argument does not transfer to the classical Riemann Hypothesis.
\end{remark}

\section{Numerical illustration}\label{sec:numerics}

No numerical computation is used in the proof. For an independent illustration, one may evaluate
$\widehat\xi$ rapidly through an incomplete-gamma expansion. Let $w=s-1/2$, $a_{n}=\pi n^2$, and denote by
$\Gamma(\alpha,x)$ the upper incomplete gamma function. Termwise integration in~\eqref{eq:def} gives
\begin{equation}
\widehat\xi\!\left(\frac12+w\right)
=2\sum_{n=1}^\infty\bigl(G_n(w)-G_n(-w)\bigr),
\label{eq:gamma-series}
\end{equation}
where
\begin{equation}
G_n(w)=a_{n}^{-1/4-w/2}
\left[
\Gamma\!\left(\frac94+\frac w2,a_{n}\right)
-\frac32\Gamma\!\left(\frac54+\frac w2,a_{n}\right)
\right].
\label{eq:Gn}
\end{equation}
The series is locally uniformly convergent and is convenient for high-precision evaluation. For example,
it gives the off-critical-line zero
\begin{equation}
\begin{split}
s_0={}&8.7693851018082313598034501394770732244884\\
&{}+13.2105919594647796864361205053041683401445\,i,
\end{split}
\label{eq:numroot}
\end{equation}
to the displayed digits. Using 80-digit arithmetic, the digits in~\eqref{eq:numroot} are stable when the truncation is increased from
$n=10$ to $n=12$; for the unrounded numerical iterate, the $n=12$ truncated series has residual below
$10^{-80}$. This numerical value is included only as a check on the analytic conclusion of
Theorem~\ref{thm:main}; it is not part of the proof.

The symmetries
\[
\widehat\xi(s)=-\widehat\xi(1-s),
\qquad
\overline{\widehat\xi(s)}=\widehat\xi(\overline{s})
\]
then generate the corresponding symmetric zeros from~\eqref{eq:numroot}.

\section{Concluding observation}

The conjectured real-rootedness fails for a structural reason that is visible before any zero search is
performed: the sine transform has a nonvanishing endpoint term and hence only finitely many real zeros,
whereas its finite-order entire-function growth forces infinitely many zeros in the plane. In particular,
the failure is not a delicate numerical phenomenon and is unrelated to the truth or falsity of the
Riemann Hypothesis.

\section*{Acknowledgments}

The authors acknowledge OpenAI's GPT-6 Astra for substantial assistance in the development of the argument presented in this paper, including the identification of the sine-transform formulation and the endpoint asymptotic leading to the disproof. The proof and exposition were also subjected to independent model-based checks using several large language models, including different versions of GPT-6 and Anthropic's Claude Opus 5.5. These additional checks were used to test the derivations, identify possible gaps, and improve the presentation. The authors independently verified the mathematical arguments and take full responsibility for the results and exposition.

\end{document}